\documentclass[reqno, a4paper,12pt]{amsart}
\usepackage{amsmath,amssymb}
\usepackage{amsthm}
\usepackage[abbrev]{amsrefs}
\usepackage{comment}

\theoremstyle{plain}
  \newtheorem{thm}{Theorem}[section]
  \newtheorem{mthm}[thm]{Main Theorem}
  \newtheorem{lem}[thm]{Lemma}
  
  \newtheorem{cor}[thm]{Corollary}
  \newtheorem{prop}[thm]{Proposition}
  \newtheorem{claim}[thm]{Claim}

\theoremstyle{definition}
  \newtheorem{dfn}[thm]{Definition}
  
  \newtheorem{ex}[thm]{Example}

\theoremstyle{remark}
  \newtheorem{rem}[thm]{Remark}

\newcommand{\Lip}{\mathcal{L}{\it ip}}
\newcommand{\id}{{\rm id}}
\newcommand{\N}{\mathbb{N}}
\newcommand{\Z}{\mathbb{Z}}

\newcommand{\R}{\mathbb{R}}

\newcommand{\A}{\mathcal{A}}
\newcommand{\B}{\mathcal{B}}

\newcommand{\F}{\mathcal{F}}
\newcommand{\G}{\mathbb{G}}

\newcommand{\LL}{\mathcal{L}}
\newcommand{\M}{\mathcal{M}}

\renewcommand{\P}{\mathcal{P}}

\newcommand{\X}{\mathcal{X}}

\newcommand{\1}{\mathbf{1}}
\newcommand{\al}{\alpha}
\newcommand{\bal}{\boldsymbol{\alpha}}

\newcommand{\de}{\delta}
\newcommand{\ep}{\varepsilon}
\newcommand{\bep}{\boldsymbol{\ep}}
\newcommand{\ka}{\kappa}
\newcommand{\io}{\iota}
\newcommand{\ph}{\varphi}
\newcommand{\la}{\lambda}

\newcommand{\umu}{\underline{\mu}}

\newcommand{\tX}{\tilde{X}}

\newcommand{\concto}{\xrightarrow{{\rm conc}}}
\newcommand{\Boxto}{\xrightarrow{\Box}}

\DeclareMathOperator{\dKF}{{\it d}_{{\rm KF}}}
\DeclareMathOperator{\dKFH}{{\it d}_{\it H}^{{\rm KF}}}
\DeclareMathOperator{\dconc}{{\it d}_{{\rm conc}}}

\DeclareMathOperator{\Gr}{Gr}

\DeclareMathOperator{\Aff}{Aff}
\DeclareMathOperator{\diam}{diam}
\DeclareMathOperator{\supp}{supp}

\DeclareMathOperator{\grad}{grad}

\DeclareMathOperator{\Sep}{Sep}
\DeclareMathOperator{\Obsdiam}{ObsDiam}
\DeclareMathOperator{\udiam}{\underline{diam}}
\DeclareMathOperator{\uObsdiam}{\underline{ObsDiam}}
  
\numberwithin{equation}{section}
  
\begin{document}

\title[Observable diameter and positive measure point]{The observable diameter of metric measure spaces and the existence of points of positive measures}

\author{Shun Oshima}
\address{Mathematical Institute, Tohoku University,
Sendai 980-8578, Japan}
\email{shun.oshima.s3@dc.tohoku.ac.jp}

\begin{abstract}
A metric measure space is a metric space with a Borel measure. 
In Gromov's theory of metric measure spaces, there are important invariants called the partial diameter and the observable diameter. We obtain the result that the partial diameter or the observable diameter equals zero if and only if there exists a point that has positive measure.
\end{abstract}

\date{\today}


\maketitle
\section{introduction}
A metric measure space $(X,d_X,\mu_X)$ is a metric space $(X,d_X)$ with a Borel measure $\mu_X$ on $X$. It is defined as a generalization of a Riemannian manifold with the Riemannian distance and the volume measure constructed by a Riemannian metric. This concept was considered to be derived from the fact that if a sequence of Riemannian manifolds converges to some space, the limit is not necessarily a Riemannian manifold. By this generalization, many of the concepts and properties considered for Riemannian manifolds can also be considered for metric measure spaces. For example, the theory of Sobolev spaces on metric measure spaces and the theory of curvature dimension conditions, which formulates the condition that the Ricci curvature is bounded from below on a metric measure space, have been studied in Riemannian geometry in recent years. (e.g. \cite{AGS}, \cite{Kazukawa-Ozawa-Suzuki}, \cite{Oshima}). 

Currently, there are several definitions of metric measure space. One of them is a metric measure space such that $\mu_X$ is a probability measure as considered by Gromov \cite{Gromov}, which is called an mm-space. The theory of mm-spaces mainly deals with the convergence of a sequence of mm-spaces called the $\Box$-convergence and the concentration. In particular, it is known that the concentration is the weakest convergence among the currently known convergences of a sequence of metric measure spaces, which is derived from the concentration of measure phenomenon. Here, the concentration of measure phenomenon is a phenomenon in that measures are distributed unevenly in high-dimensional spaces, and was discovered and formalized by L\'{e}vy \cite{Levy} and Milman \cite{V.D.Milman}. This phenomenon can also be expressed in terms of functions as ``Any $1$-Lipschitz function on a high-dimensional space is measurably close to a constant function, that is, $1$-Lipschitz function on a one-point space". The concentration is a convergence considered so that the above can be interpreted as a high-dimensional space is close to a one-point space. 

In Gromov's theory, some of the important invariants of metric measure spaces are the $\al$-partial diameter and the $\al$-observable diameter for $\al\in (0,1)$. One of the properties of these invariants is that the $\al$-partial diameter (resp. the $\al$-observable diameter) of a sequence of mm-spaces converges to zero for any $\al\in (0,1)$ if and only if this sequence $\Box$-converges (resp. concentrates) to a one-point space. Moreover, it is known that the convergence of these two types of diameters to zero is not equivalent (see Example \ref{ex}, Remark \ref{rem}).

The main theorem of this paper is the following theorem which gives the equivalence condition for the $\al$-partial diameter and the $\al$-observable diameter of mm-space to be zero.
\begin{mthm}
\label{atomObs}
Let $X$ be an mm-space and let $\al \in (0,1)$. Then, the following are equivalent.
\begin{enumerate}
\item There exists a point $x\in X$ such that $\mu_X(\{x\})\ge \al$.
\item $\diam(X; \al)=0$.
\item $\Obsdiam(X;\al)=0$.
\end{enumerate}
\end{mthm}
This main theorem shows that these two types of diameters coinciding with zero for an mm-space and $\al\in (0,1)$ are equivalent. 

Furthermore, Main Theorem \ref{atomObs} is extended for the multivariable partial diameter and the multivariable observable diameter, which are concepts newly defined in this paper.
\begin{mthm}
\label{atomultiobsdiam}
Let $X$ be an mm-space and let $\bal=(\al_1,\ldots,\al_n)\in (0,\infty)^n$ with $n\in \N\cup\{\infty\}$ where $(0,\infty)^{\infty}=(0,\infty)^{\N}$. Then, the following are equivalent.
\begin{enumerate}
\item There exist an mm-space $Y$ with $Y\succ X$ and distinct points $\{y_i\}_{i=1}^n\subset Y$ such that $\mu_Y(\{y_i\})\ge \al_i$ for any $i\in \N\cap [1,n]$.
\item There exists $\{x_i\}_{i=1}^n \subset X$ such that 
\begin{equation*}
\mu_X(\{x_i\mid i\in I\})\ge \sum_{i\in I}\al_i 
\end{equation*}
for any $I\subset \N\cap [1,n]$.
\item $\udiam(X;\bal)=0$.
\item $\uObsdiam(X;\bal)=0$.
\end{enumerate} 
\end{mthm}
Finally, we describe the structure of this paper. In Section \ref{Preliminaries}, we state the definition and basic properties of mm-spaces.  In Section \ref{multi}, we define the multivariable partial (observable) diameter of mm-spaces and consider their properties. In Section \ref{ProofMT}, we give the proof of Main Theorem \ref{atomultiobsdiam}. In Section \ref{Propmultidiam}, we give some results of the multivariable partial (observable) diameter that are not directly related to the main theorem. In Section \ref{andef}, we give another definition of the multivariable partial and observable diameters and prove an analog of Main Theorem \ref{atomultiobsdiam} for these multivariable partial and observable diameters.

\section*{Notations}
We list the notations we will use throughout this paper.
\begin{itemize}
\item $\B_X$ denotes the set of all Borel subsets of a topological space $X$.
\item $\P(X)$ denotes the set of all Borel probability measures on $X$.
\item $\Lip_L(X)$ denotes the set of all $L$\,-Lipschitz functions on $X$.
\item $\mu\otimes\nu$ denotes the product measure of two measures $\mu$ and $\nu$. 
\item $\1_A$ denotes the characteristic function of a subset $A\subset X$.
\item For $x\in X$, $\de_x$ denotes the Dirac measure, namely, $\de_x$ is a probability measure on $X$ defined as $\de_x(A):=\1_A(x)$ for any measurable set $A$.
\item $\supp \mu$ denotes the support of a Borel measure $\mu$ on a topological space $X$.
\item $B_X(x,r)$ denotes the open ball in a metric space $(X,d_X)$ with center $x\in X$ and radius $r>0$.
\item For a point $x$ in a metric space $(X,d_X)$ and a subset $A$ of $X$, we write $d_X(x,A):=\inf_{a\in A}d_X(x,a)$.
\item For two Borel subsets $A$ and $B$ of a metric space $(X,d_X)$, we write 
\begin{equation*}
d_X(A,B):=\inf_{x\in A,\ y\in B}d_X(x,y).
\end{equation*}
\item $N_\ep(A)$ denotes the $\ep$-neighborhood of a subset $A$ of a metric space $(X,d_X)$, namely,
\begin{equation*}
N_\ep(A):=\bigcup_{a\in A}B_X(a,\ep).
\end{equation*}
\item For $p\in [1,\infty]$ and $n\in\N\cup\{\infty\}$, $\|x\|_p$ denotes the $\ell^p$-norm of $x=(x_1,\ldots,x_n)\in \R^n$, namely,
\begin{equation*}
\|x\|_p:=
\begin{cases}
\displaystyle \left(\sum_{i=1}^n|x_i|^p\right)^{\frac{1}{p}}\qquad&(p<\infty)\\
\displaystyle \sup_{1\le i\le n}|x_i|\qquad&(p=\infty)
\end{cases}
.
\end{equation*}
\item For $x,y\in \R^n$, $\langle x,y\rangle$ denotes the Euclidean inner product of $x$ and $y$.
\end{itemize}
\section{Preliminaries}
\label{Preliminaries}

In this section, we give the definition and basic properties of mm-spaces.
\begin{dfn}[mm-space]
A triple $(X,d_X,\mu_X)$ is called an \emph{mm-space} if $(X,d_X)$ is a complete separable metric space and $\mu_X$ is a Borel probability measure on $X$.
\end{dfn}
\begin{dfn}[mm-isomorphism]
Two mm-spaces $(X,d_X,\mu_X)$ and $(Y,d_Y,\mu_Y)$ are said to be \emph{mm-isomorphic} if there exists an isometry $f:\supp \mu_X\to \supp \mu_Y$ such that the push-forward measure $f_*\mu_X$ of $\mu_X$ by $f$ equals $\mu_Y$. Such an $f$ is called an \emph{mm-isomorphism}.
\end{dfn}
Let $\X$ denote the set of mm-isomorphism classes of mm-spaces. 
\begin{dfn}[Lipschitz order]
Let $X$ and $Y$ be two mm-spaces. We say that $X$ (\emph{Lipschitz}) \emph{dominates} $Y$ and write $Y \prec X$ (or $X\succ Y$) if there exists a $1$-Lipschitz map $f:X\to Y$ satisfying $f_*\mu_X = \mu_Y$. We call such an $f$ a \emph{dominating map}, and we call the relation $\prec$ on $\X$ the \emph{Lipschitz order}.
\end{dfn}
\begin{dfn}[Parameter]
Let $I:=[0,1)$ and let $X$ be an mm-space. A Borel map $\ph:I\to X$ is called a \emph{parameter of $X$} if $\ph$ satisfies $\ph_*\LL^1=\mu_X$, where $\LL^1$ is the Lebesgue measure on $I$.
\end{dfn}
\begin{prop}[{\cite[Lemma 4.2]{ShioyaMMG}}]
Any mm-space has a parameter. 
\end{prop}
\begin{dfn}[Box distance between mm-spaces]
The \emph{box distance} $\Box(X,Y)$ between two mm-spaces $X$ and $Y$ is defined by the infimum of $\ep>0$ satisfying that there exist a Borel set $\tilde{I}\subset I$ and parameters $\ph:I\to X$ and $\psi:I\to Y$ of $X$ and $Y$ such that
\begin{enumerate}
\item $\LL^1(\tilde{I})\ge1-\ep$,
\item $|d_X(\ph(s),\ph(t))-d_Y(\psi(s),\psi(t))|\le\ep$\qquad for any $s,t\in \tilde{I}$.
\end{enumerate}
\end{dfn}
It is known that $\Box$ is a metric on $\X$ and $(\X,\Box)$ is a complete metric space \cite[Theorem 4.10, Theorem 4.14]{ShioyaMMG}. If a sequence $\{X_n\}_{n\in \N}$ of mm-spaces converges to an mm-space $Y$ with respect to $\Box$, then we say that $\{X_n\}_{n\in \N}$  \emph{$\Box$-converges to $Y$} and write $X_n\Boxto Y$. The topology on $\X$ induced by $\Box$ is called the \emph{box topology}.
\begin{prop}[{\cite[Proposition 4.12]{ShioyaMMG}}]
\label{Boxwconv}
Let $X$ be a complete separable metric space. For any two Borel probability measures $\mu$ and $\nu$ on $X$, we have
\begin{equation*}
\frac{1}{2}\Box((X,\mu),(X,\nu))\le \Box((2^{-1}X,\mu),(2^{-1}X,\nu))\le d_P(\mu,\nu),
\end{equation*}
where $d_P$ is the Prokhorov metric on $\P(X)$. In particular, if a sequence $\{\mu_n\}_{n\in \N}$ in $\P(X)$ converges weakly to some $\mu\in \P(X)$, then $(X,\mu_n)$ $\Box$-converges to $(X,\mu)$.
\end{prop}
\begin{dfn}[Observable diameter]
Let $X$ be an mm-space and let $\al\ge 0$. We define the \emph{$\al$-partial diameter $\diam(X;\al)$ of $X$} by
\begin{equation*}
\diam(X;\al)=\diam(\mu_X;\al):=\inf\{\diam A\mid A\in \B_X \text{ with } \mu_X(A)\ge \al\},
\end{equation*}
and we define 
\begin{align*}
\Obsdiam(X;\al):=&\sup\{\diam(f_*\mu_X;\al)\mid f\in \Lip_1(X)\},\\
\Obsdiam(X):=&\inf_{\al\in (0,1)}((1-\al)\lor \Obsdiam(X;\al)).
\end{align*}
We call $\Obsdiam(X;\al)$ (resp. $\Obsdiam(X)$) the \emph{$\al$-observable diameter of $X$} (resp. \emph{observable diameter of $X$}).
\end{dfn}
\begin{rem}
The notation of $\al$-observable diameter in this paper is different from the previously known notation, e.g. in \cite{ShioyaMMG}, in order to be consistent with the multivariable observable diameter, which will be defined in Section \ref{multi}.
\end{rem}
\begin{prop}[{\cite[Proposition 2.18]{ShioyaMMG}}]
\label{Obsdiamonotone}
Let $X$ and $Y$ be two mm-spaces and let $\al\in (0,1)$.
\begin{enumerate}
\item If $X$ is dominated by $Y$, then
\begin{equation*}
\diam(X;\al) \le \diam(Y;\al).
\end{equation*}
\item We have
\begin{equation*}
\Obsdiam(X;\al) \le \diam(X;\al).
\end{equation*}
\item If $X$ is dominated by $Y$, then
\begin{equation*}
\Obsdiam(X;\al) \le \Obsdiam(Y;\al).
\end{equation*}
  \end{enumerate}
\end{prop}
\begin{prop}
\label{box*}
Let $*$ be an mm-space consisting of one point, i.e. $*:=(\{*\}, d_*, \de_*)$. Then, a sequence $\{X_n\}_{n\in \N}$ of mm-spaces $\Box$-converges to $*$ if and only if 
\begin{equation*}
\lim_{n\to\infty}\diam(X_n;\al)=0
\end{equation*} 
for any $\al\in (0,1)$.
\end{prop}
\begin{dfn}[L\'{e}vy family]
A sequence $\{X_n\}_{n\in \N}$ of mm-spaces is called a \emph{L\'{e}vy family} if
\begin{equation*}
\lim_{n\to \infty}\Obsdiam(X_n)=0,
\end{equation*}
or equivalently
\begin{equation*}
\lim_{n\to \infty}\Obsdiam(X_n;\al)=0
\end{equation*}
for any $\al\in (0,1)$.
\end{dfn}
\begin{dfn}[Ky Fan metric]
The \emph{Ky Fan distance} $\dKF(f,g)$ between two measurable functions $f,g:\Omega \to \R$ on a probability measure space $(\Omega, \mu)$ is defined by 
\begin{equation*}
\dKF(f,g):=\inf\left\{\ep\ge0\mid \mu(\{x\in \Omega\mid |f(x)-g(x)|>\ep\})\le \ep\right\}.
\end{equation*}
This distance function $\dKF$ is called the \emph{Ky Fan metric}. 
\end{dfn}
\begin{prop}[cf. {\cite[Lemma 1.26]{ShioyaMMG}}]
\label{Prolem1}
Let $f,g:X\to \R$ be two Borel functions on an mm-space $X$. Then, we have
\begin{equation*}
d_P(f_*\mu_X,g_*\mu_X)\le \dKF(f,g).
\end{equation*}
\end{prop}
\begin{prop}[cf. {\cite[Corollary 4.37]{ShioyaMMG}}]
\label{Prolem2}
Let $f:X\to \R$ be a $1$-Lipschitz function on a metric space $X$. For any $\mu,\nu\in \P(X)$, we have
\begin{equation*}
d_P(f_*\mu,f_*\nu)\le d_P(\mu,\nu).
\end{equation*}
\end{prop}
\begin{dfn}[Observable distance]
For two mm-spaces $X$ and $Y$, The \emph{observable distance} $\dconc(X,Y)$ between $X$ and $Y$ is defined as 
\begin{equation*}
\dconc(X,Y):=\inf_{\ph, \psi}\dKFH\left(\ph^*\Lip_1(X),\psi^*\Lip_1(Y)\right),
\end{equation*}
where $\ph$ and $\psi$ run over all parameters of $X$ and $Y$ respectively, $\ph^*\Lip_1(X):=\{f\circ\ph\mid f\in \Lip_1(X)\}$ which is a subset of the set of all measurable functions on $(I, \LL^1)$, and $\dKFH$ is the Hausdorff distance with respect to the Ky Fan metric.
\end{dfn}
It is also known that $\dconc$ is a metric on $\X$ ({\cite[Theorem 5.13]{ShioyaMMG}}). If a sequence $\{X_n\}_{n\in \N}$ of mm-spaces converges to an mm-space $Y$ with respect to $\dconc$, then we say that \emph{$\{X_n\}_{n\in \N}$ concentrates to $Y$} and write $X_n\concto Y$.
\begin{prop}[{\cite[Proposition 5.5]{ShioyaMMG}}]
\label{Boxconc}
For any $X,Y\in \X$, we have $\dconc(X,Y)\le \Box(X,Y)$. In particular, if a sequence $\{X_n\}_{n\in \N}$ of mm-spaces $\Box$-converges to an mm-space $Y$, then $\{X_n\}_{n\in \N}$ concentrates to $Y$.
\end{prop}
The topology on $\X$ induced by $\dconc$ is called the \emph{concentration topology}.
\begin{prop}[{\cite[Proposition 5.7, Corollary 5.8]{ShioyaMMG}}]
For any mm-space $X$, we have
\begin{equation*}
\dconc(X,*)\le \Obsdiam(X)\le 2\dconc(X,*).
\end{equation*}
In particular, a sequence $\{X_n\}_{n\in \N}$ of mm-spaces concentrates to $*$ if and only if $\{X_n\}_{n\in \N}$ is a L\'{e}vy family.
\end{prop}
\begin{ex}[{\cite[Section 2]{ShioyaMMG}}]
\label{ex}
Let $S^n(r_n)$ be the $n$-dimensional sphere of radius $r_n>0$. We consider $S^n(r_n)$ with the Riemannian distance and the normalized Riemannian volume measure as an mm-space. Then, for any $\al\in (0,1)$,
\begin{equation*}
\lim_{n\to\infty}\Obsdiam(S^n(r_n);\al)=0\ \Leftrightarrow\ \lim_{n\to \infty}\frac{r_n}{\sqrt{n}}=0.
\end{equation*}
In particular, $\{S^n(1)\}_{n\in \N}$ is a L\'{e}vy family.
\end{ex}
\begin{rem}
\label{rem}
For any $\al\in (0,1)$,
\begin{equation*}
\liminf_{n\to \infty}\diam(S^n(1);\al)\ge \frac{\pi}{2}.
\end{equation*}
In particular, $S^n(1)$ does not $\Box$-converge to $*$ by Proposition \ref{box*}. 

Moreover, \cite[Corollary 5.20]{ShioyaMMG} says that $\{S^n(1)\}_{n\in \N}$ has no $\Box$-convergent subsequence.
\end{rem}
\begin{prop}[{\cite[Section 4.4]{ShioyaMMG}}]
\label{findimappro}
Let $X$ be an mm-space. For any $m\in \N$, there exists $\umu_m\in \P(\R^m)$ such that
\begin{align*}
X_1\prec X_2\prec\cdots\prec X_m\prec\cdots \prec X\text{\qquad and \qquad} X_m\Boxto X\qquad(m\to\infty),
\end{align*}
where $X_m:=(\R^m,\|\cdot\|_{\infty},\umu_m)$.
\end{prop}
\begin{dfn}[$\ep$-mm-Isomorphism] 
Let $X$ and $Y$ be two mm-spaces and let $\ep \ge 0$. We say that a Borel map $f:X\to Y$ is an \emph{$\ep$-mm-isomorphism} if there exists a Borel subset $\tX \subset X$ such that
\begin{enumerate}
\item $\mu_X(\tX) \ge 1-\ep$,
\item $|d_X(x,y) - d_Y(f(x),f(y))| \le \ep$\qquad for any $x,y \in \tX$,
\item $d_P(f_*\mu_X,\mu_Y) \le \ep$.
\end{enumerate}
\end{dfn}
\begin{prop}[{\cite[Lemma 4.22]{ShioyaMMG}}]
\label{epmm}
Let $X$ and $Y$ be two mm-spaces and let $\ep \ge 0$.
\begin{enumerate}
\item If there is an $\ep$-mm-isomorphism $f:X\to Y$, then $\Box(X,Y) \le 3\ep$.
\item If $\Box(X,Y) < \ep$, then there exists a $3\ep$-mm-isomorphism $f:X\to Y$.
\end{enumerate}
In particular, $X_N$ $\Box$-converges to $Y$ if and only if there exist $\ep_N\searrow 0$ and an $\ep_N$-mm-isomorphism $f_N:X_N\to Y$.
\end{prop}
\begin{prop}[cf. {\cite[Lemma 5.4]{ShioyaMMG}}]
\label{Lipappro}
Let $X$ be an mm-space and let $\ep>0$. If a Borel map $f:X\to \R$ satisfies that there exists a Borel subset $X_0\subset X$ such that 
\begin{enumerate}
\item $\mu_X(X_0) \ge 1-\ep$,
\item $d_Y(f(x),f(y)) \le d_X(x,y) + \ep$\qquad for any $x,y \in X_0$,
\end{enumerate}
then there exists a $1$-Lipschitz function $\tilde{f}:X\to \R$ such that $\dKF(f,\tilde{f})\le \ep$.
\end{prop}
\begin{prop}
\label{mmsqueeze}
Let $\{X_n\}_{n\in \N}$, $\{Y_n\}_{n\in \N}$, and $\{Z_n\}_{n\in \N}$ be three sequences of mm-spaces satisfying $X_n\prec Y_n\prec Z_n$ for any $n\in \N$. If $X_n$ and $Z_n$ $\Box$-converge to some mm-space $X$ as $n\to\infty$, then $Y_n$ also $\Box$-converges to $X$ as $n\to\infty$.
\end{prop}
\begin{proof}
$Z_n\Boxto X$ and $Y_n\prec Z_n$ imply that $\{Y_n\}_{n\in \N}$ is $\Box$-precompact by {\cite[Lemma 4.28]{ShioyaMMG}}. Hence, for any subsequence $\{Y_{n_k}\}_{k\in \N}$ of $\{Y_n\}_{n\in \N}$, there exist a subsequence $\{Y_{n_k'}\}_{k\in \N}$ of $\{Y_{n_k}\}_{k\in \N}$ and an mm-space $Y$ such that $Y_{n_k'}\Boxto Y$ as $k\to \infty$. Then, since $X_{n_k'}$ and $Z_{n_k'}$ $\Box$-converge to $X$, and $X_{n_k'}\prec Y_{n_k'}\prec Z_{n_k'}$, it follows from {\cite[Theorem 4.35]{ShioyaMMG}} that $X\prec Y\prec X$. Thus, $X$ is mm-isomorphic to $Y$ by {\cite[Proposition 2.11]{ShioyaMMG}}. Thus, since any subsequence of $\{Y_n\}_{n\in \N}$ has a subsequence $\Box$-converging to $X$, $Y_n$ $\Box$-converges to $X$.
\end{proof}
\begin{dfn}
Let $(X,d_X)$ be a metric space and let $\F(X)$ be the set of all closed subsets of $X$. We say that a sequence $\{A_n\}_{n\in \N}$ in $\F(X)$ \emph{converges weakly to $A\in \F(X)$} if the following (1) and (2) are satisfied.
\begin{enumerate}
\item For any $x\in A$,
\begin{equation*}
\lim_{n\to\infty}d_X(x,A_n)=0.
\end{equation*}
\item For any $x\in X\setminus A$,
\begin{equation*}
\liminf_{n\to\infty}d_X(x,A_n)>0.
\end{equation*}
\end{enumerate}
Here, we define $d_X(x,\emptyset):=\infty$ for any $x\in X$.
\end{dfn}
\begin{prop}[{\cite[Lemma 6.6]{ShioyaMMG}}]
\label{wHsub}
If $(X,d_X)$ is separable, then any sequence in $\F(X)$ has a weakly convergent subsequence.
\end{prop}
\begin{prop}[cf. {\cite[Lemma 6.4]{Nakajima-Shioya}}]
\label{wHconti}
Let $(X,d_X)$ be a complete separable metric space, let $\{A_n\}_{n\in \N}$ be a sequence in $\F(X)$, and let $\{\mu_n\}_{n\in \N}$ be a sequence in $\P(X)$. If $A_n$ converges weakly to $A\in \F(X)$ and $\mu_n$ converges weakly to $\mu\in \P(X)$, then we have
\begin{align*}
\diam A&\le \liminf_{n\to \infty}\diam A_n,\\
\mu(A)&\ge \limsup_{n\to\infty}\mu_n(A_n).
\end{align*}
\end{prop}
\section{multivariable observable diameter of mm-spaces}
\label{multi}

In this section, we define the multivariable partial (observable) diameter of mm-spaces and consider their properties.
\begin{dfn}
We define $\A$ as 
\begin{align*}
\A&:=\bigcup_{n\in \N\cup\{\infty\}}(0,\infty)^n,
\end{align*}
where $(0,\infty)^{\infty}=(0,\infty)^{\N}$. For $\bal\in \A$, $n(\bal)$ is a unique element in $\N\cup\{\infty\}$ with $\bal\in (0,\infty)^{n(\bal)}$. For $n\in \N\cup\{\infty\}$, $\N_n$ is defined as
\begin{equation*}
\N_n:=\N\cap [1,n].
\end{equation*}
\end{dfn}
In the following, $n(\bal)$ for $\bal\in \A$ is denoted by $n$ unless otherwise noted. 
\begin{dfn}[multivariable observable diameter]
Let $X$ be an mm-space and let $\bal\in \A$. We define the \emph{$\bal$-partial diameter $\diam(X;\bal)$, $\udiam(X;\bal)$ of $X$}, and the \emph{$\bal$-observable diameter $\uObsdiam(X;\bal)$ of $X$} by
\begin{align*}
\diam(X;\bal)&:=
\inf\left\{\left.\sup_{i\in \N_n}\diam A_i\ \right| \ \{A_i\}_{i=1}^n\in D_X(\bal)\right\},\\
\udiam(X;\bal)&:=\inf_{Y\succ X}\diam(Y;\bal),\\
\uObsdiam(X;\bal)&:=\sup\{\udiam((\R, f_*\mu_X);\bal)\mid f\in \Lip_1(X)\},
\end{align*}
where $D_X(\bal)$ is the set of families $\{A_i\}_{i=1}^n$ of disjoint Borel subsets of $X$ with $\mu_X(A_i)\ge \al_i$ for any $i\in \N_n$. In addition, we define $\diam(X;\bal):=\infty$ if $D_X(\bal)=\emptyset$.
\end{dfn}
\begin{rem}
In Section \ref{andef}, we introduce definitions of the multivariable partial and observable diameters for the case where $\diam(X;\bal):=0$ if $D_X(\bal)=\emptyset$.
\end{rem}
These multivariable partial and observable diameters have monotonicity similar to Proposition \ref{Obsdiamonotone}.
\begin{prop}[cf. {\cite[Proposition 2.18]{ShioyaMMG}}]
\label{multiobsdiamonotone}
Let $X$ and $Y$ be two mm-spaces and let $\bal\in \A$.
\begin{enumerate}
\item If $X$ is dominated by $Y$, then
\begin{equation*}
\udiam(X;\bal) \le \udiam(Y;\bal).
\end{equation*}
\item We have
\begin{equation*}
\uObsdiam(X;\bal) \le \udiam(X;\bal).
\end{equation*}
\item If $X$ is dominated by $Y$, then
\begin{equation*}
\uObsdiam(X;\bal) \le \uObsdiam(Y;\bal).
\end{equation*}
\end{enumerate}
\end{prop}
\begin{rem}
If $n=1$, then we have $\udiam(X;\bal)=\diam(X;\bal)$. In particular, $\diam(X;\bal)$ is monotone with respect to the Lipschitz order. However, if $n\ge 2$, then $\diam(X;\bal)$ does not have the monotonicity with respect to the Lipschitz order. $\udiam(X;\bal)$ is a modification of $\diam(X;\bal)$ to obtain the monotonicity with respect to the Lipschitz order. Furthermore, if we define $\Obsdiam(X;\bal)$ as 
\begin{equation*}
\Obsdiam(X;\bal):=\sup\{\diam((\R, f_*\mu_X);\bal)\mid f\in \Lip_1(X)\}
\end{equation*}
in the same way as for $n=1$, then $\Obsdiam(X;\bal)$ is identically equal to infinity when $n\ge 2$ by $\diam((\R, f_*\mu_X);\bal)=\infty$ if $f$ is a constant function.
\end{rem}
\begin{dfn}
For an mm-space $X$ and $\bal\in \A$, we define $\underline{D_X}(\bal)$ as the set of triples $(Y,F,\{A_i\}_{i=1}^n)$ of an mm-space $Y$ dominating $X$, a dominating map $F:Y\to X$, and $\{A_i\}_{i=1}^n\in D_Y(\bal)$.
\end{dfn}
\begin{rem}
\label{diamrem}
For an mm-space $X$ and $\bal\in \A$, we have
\begin{equation*}
\udiam(X;\bal)=\inf\left\{\left.\sup_{i\in \N_n}\diam A_i\ \right|\  (Y,F,\{A_i\}_{i=1}^n)\in \underline{D_X}(\bal)\right\}.
\end{equation*}
\end{rem}
\begin{dfn}
For an mm-space $X$ and $\bal\in \A$, we define $\M_X(\bal)$ as the set of a family $\{\mu_i\}_{i=1}^n$ of Borel probability measures on $X$ such that 
\begin{equation*}
\mu_X-\sum_{i=1}^n\al_i\mu_i
\end{equation*}
is a non-negative measure on $X$.
\end{dfn}
\begin{prop}
For an mm-space $X$ and $\bal\in \A$, we have
\begin{equation*}
\udiam(X;\bal)=\inf\left\{\left.\sup_{i\in \N_n}\diam \supp \mu_i\ \right|\ \{\mu_i\}_{i=1}^n\in \M_X(\bal) \right\}.
\end{equation*}
\end{prop}
\begin{proof}
We put the right-hand side of the equality as $R$. First, we prove $\udiam(X;\bal)\le R$. We take any $\{\mu_i\}_{i=1}^n\in \M_X(\bal)$ and set $\nu_X:=
\mu_X-\sum_{i=1}^n\al_i\mu_i$. If we define an mm-space $Y$ as 
\begin{align*}
Y&:=X\times (\N_n\cup\{0\}),\\
d_Y&:Y\times Y\to[0,\infty),\\
&d_Y((x,i),(y,j)):=d_X(x,y)+|i-j|,\\
\mu_Y&:=\nu_X\otimes \de_0+\sum_{i=1}^n\al_i\mu_i\otimes \de_0,
\end{align*}
a $1$-Lipschitz map $F:Y\to X$ as $F(x,i):=x$, and a family $\{A_i\}_{i=1}^n$ of Borel subsets of $Y$ as $A_i:=\supp \mu_i\times \{i\}$, then we have $(Y,F,\{A_i\}_{i=1}^n)\in \underline{D_X}(\bal)$ and 
\begin{align*}
\udiam(X;\bal)\le \diam(Y;\bal)\le \sup_{i\in \N_n}\diam A_i=\sup_{i\in \N_n}\diam \supp\mu_i.
\end{align*} 
By the arbitrariness of $\{\mu_i\}_{i=1}^n\in \M_X(\bal)$, we obtain $\udiam(X;\bal)\le R$.

Next, we prove $\udiam(X;\bal)\ge R$. For any $(Y,F, \{A_i\}_{i=1}^n)\in \underline{D_X}(\bal)$ and $i\in\N_n$, we define a probability measure $\mu_i$ as $\mu_i:=F_*(\mu_Y)_{A_i}$ where
\begin{equation*}
(\mu_Y)_{A_i}:=\frac{\mu_Y(\cdot\cap A_i)}{\mu_Y(A_i)}\in \P(Y).
\end{equation*}
Then, we have $\{\mu_i\}_{i=1}^n\in \M_X(\bal)$ since for any $B\in \B_X$, 
\begin{align*}
\sum_{i=1}^n\al_i\mu_i(B)&= \sum_{i=1}^n\frac{\al_i}{\mu_Y(A_i)}\mu_Y(F^{-1}(B)\cap A_i)\\
&\le \mu_Y\left(F^{-1}(B)\cap \bigcup_{i=1}^nA_i\right)\\
&\le F_*\mu_Y(B)=\mu_X(B).
\end{align*} 
In addition, $\supp \mu_i\subset \overline{F(A_i)}$ from
\begin{equation*}
\mu_i(\overline{F(A_i)})\ge (\mu_Y)_{A_i}(A_i)=1.
\end{equation*}
Thus, we have
\begin{equation*}
R\le \sup_{i\in\N_n}\diam \supp \mu_i\le \sup_{i\in\N_n}\diam \overline{F(A_i)}\le \sup_{i\in\N_n}\diam A_i.
\end{equation*}
The arbitrariness of $(Y,F, \{A_i\}_{i=1}^n)\in \underline{D_X}(\bal)$ implies $\udiam(X;\bal)\ge R$.
\end{proof}
\begin{lem}
\label{convlem}
Let $X$ be an mm-space and let $\bal\in \A$ with $\|\bal\|_1\le 1$. Suppose that there exist a sequence $\{\mu_m\}_{m\in \N}$ in $\P(X)$, a sequence $\{\ep_m\}_{m\in \N}$ converging to $0$, and a sequence $\{\{B_{i,m}\}_{i=1}^n\}_{m\in \N}$ of families of closed subsets of $X$ such that 
\begin{itemize}
\item $\mu_m$ weakly converges to some $\mu\in \P(X)$,
\item $\sup_{i\in \N_n}\diam B_{i,m}\to 0$\qquad$(m\to\infty)$,
\item $\mu_m(\bigcup_{i\in I}B_{i,m})\ge \sum_{i\in I}(\al_i-\ep_m)$ for any $m\in \N$ and $I\subset \N_n$.
\end{itemize}
Then, there exists $\{x_i\}_{i=1}^n\subset X$ such that 
\begin{equation*}
\mu(\{x_i\mid i\in I\})\ge \sum_{i\in I}\al_i 
\end{equation*}
for any $I\subset \N_n$.
\end{lem}
\begin{proof}
For any $i\in \N_n$, it follows from Proposition \ref{wHsub} that $\{B_{i,m}\}_{m\in \N}$ has a weakly convergent subsequence. By diagonal argument, there exist a subsequence $\{m_k\}_{k\in \N}$ and a family $\{B_i\}_{i=1}^n$ of closed subsets of $X$ such that $B_{i,m_k}$ converges weakly to $B_i$ as $k\to \infty$ for any $i\in \N_n$. Since Proposition \ref{wHconti} implies 
\begin{align*}
\diam B_i&\le \liminf_{k\to\infty}\diam B_{i,m_k}=0,\\
\mu(B_i)&\ge \limsup_{k\to\infty} \mu_{m_k}(B_{i,m_k})\ge \al_i>0,
\end{align*}
for any $i\in \N_n$, there exists $x_i\in X$ such that $B_i=\{x_i\}$. We take any $I\subset \N_n$ and $\ell\in \N$ with $\ell\le n$. Then, $\bigcup_{i\in I_{\ell}}B_{i,m_k}$ converges weakly to $\bigcup_{i\in I_{\ell}}B_i=\{x_i\mid i\in I_{\ell}\}$ where $I_{\ell}:=I\cap \N_{\ell}$. Thus, Proposition \ref{wHconti} implies 
\begin{align*}
\mu\left(\{x_i\mid i\in I\}\right)&\ge\mu\left(\{x_i\mid i\in I_{\ell}\}\right)\ge \limsup_{k\to\infty}\mu_{m_k}\left(\bigcup_{i\in I_{\ell}}B_{i,m_k}\right)\\
&\ge\lim_{k\to\infty}\sum_{i\in I_{\ell}}(\al_i-\ep_{m_k})= \sum_{i\in I_{\ell}}\al_i,
\end{align*}
and the right-hand side converges to $\sum_{i\in I}\al_i$ as $\ell\nearrow n$. This completes the proof.
\end{proof}
\begin{lem}
\label{multidiamzero}
Let $X$ be an mm-space and let $\bal\in \A$. Then, the  following are equivalent.
\begin{enumerate}
\item $\diam(X;\bal)=0$.
\item There are distinct points $\{x_i\}_{i=1}^n\subset X$ such that $\mu_X(\{x_i\})\ge \al_i$ for any $i\in \N_n$.
\end{enumerate}
\end{lem}
\begin{proof}
$(2)\Rightarrow (1)$ is obvious. Hence, we prove only $(1)\Rightarrow (2)$. By (1), there exists a sequence $\{\{A_{i,m}\}_{i=1}^n\}_{m\in\N}$ in $D_X(\bal)$ such that 
\begin{equation*}
\lim_{m\to\infty}\sup_{i\in \N_n}\diam A_{i,m}=0.
\end{equation*}
Then, it follows from Lemma \ref{convlem} for $\mu_m:=\mu_X$ and $B_{i,m}:=\overline{A_{i,m}}$ that there exists $\{x_i\}_{i=1}^n\subset X$ such that $\mu_X(\{x_i\})\ge \al_i$ for any $i\in \N_n$. Finally, we prove that $\{x_i\}_{i=1}^n$ are distinct. By proof of Lemma \ref{convlem}, there exists a subsequence $\{m_k\}_{k\in \N}$ such that $\overline{A_{i,m_k}}$ converges weakly to $\{x_i\}$. In addition, for any $i,j\in \N_n$ with $i\ne j$, we can take $\ep>0$ satisfying $\mu_X(B_X(x_i,\ep)\setminus \{x_i\})<\al_i$ and $\mu_X(B_X(x_j,\ep)\setminus \{x_j\})<\al_j$ since 
\begin{equation*}
\lim_{\ep\searrow0}\mu_X(B_X(x_i,\ep)\setminus \{x_i\})=\lim_{\ep\searrow0}\mu_X(B_X(x_j,\ep)\setminus \{x_j\})=0.
\end{equation*}
Hence, there exists $N_{\ep}\in \N$ such that $k\ge N_{\ep}$ implies $A_{i,m_k}\subset B_X(x_i,\ep)$ and $A_{j,m_k}\subset B_X(x_j,\ep)$. Here, if $x_i\notin A_{i,m_k}$, then we obtain
\begin{equation*}
\al_i>\mu_X(B_X(x_i,\ep)\setminus \{x_i\})\ge \mu_X(A_{i,m_k})\ge \al_i.
\end{equation*}
It is a contradiction. Similarly, it contradicts if $x_j\notin A_{j,m_k}$. Thus, we get $x_i\in A_{i,m_k}$ and $x_j\in A_{j,m_k}$. Therefore, $A_{i,m_k}\cap A_{j,m_k}=\emptyset$ implies $x_i\ne x_j$.
\end{proof}
\section{Proof of Main Theorem}
\label{ProofMT}

The purpose in this section is to prove the main theorems. Since Main Theorem \ref{atomObs} can be proved as a corollary of Main Theorem \ref{atomultiobsdiam}, we prove only Main Theorem \ref{atomultiobsdiam}. First, we give some definitions and propositions to prove Main Theorem \ref{atomultiobsdiam}.
\begin{dfn}
Let $m$ be a positive integer. We say that $W\subset \R^m$ is an \emph{affine subspace of $\R^m$} if there exist a linear subspace $V$ of $\R^m$ and $b\in \R^m$ such that $W=V+b$, where $V+b:=\{v+b\mid v\in V\}$.
\end{dfn}
Since such a linear subspace $V$ is unique for an affine subspace $W$, we write the linear subspace $V$ as $\widetilde{W}$, and we define the dimension $\dim W$ of an affine subspace $W$ as
\begin{equation*}
\dim W:=\dim \widetilde{W}.
\end{equation*}
\begin{rem}
\label{affprop}
Only one of the following holds for two affine subspaces $W_1$ and $W_2$ with $\dim W_1\le \dim W_2$.
\begin{itemize}
\item $W_1\cap W_2=\emptyset$.
\item $W_1 \subset W_2$.
\item $\dim(W_1\cap W_2)\le \dim W_1-1$.
\end{itemize}
\end{rem}
Next lemma follows from the fact that any manifold is a Baire space.
\begin{lem}
\label{mfd}
Let $m\in \N$ and let $M$ be an $m$-dimensional manifold. If $\{M_i\}_{i\in \N}$ is a family of countable closed subsets of $M$ which are homeomorphic to manifolds whose dimensions are less than $m$, then we have
\begin{equation*}
M\ne \bigcup_{i\in \N}M_i.
\end{equation*}
\end{lem}
The following lemma and corollary are most important to prove Main Theorem \ref{atomultiobsdiam}.
\begin{lem}
\label{afflem}
Let $\mu$ be a Borel probability measure on $\R^m$ and let $A$ be a countable subset of $\R^m$. Then, there exists an $(m-1)$-dimensional linear subspace $V$ of $\R^m$ such that we have 
\begin{equation}
\label{cond}
\text{$|V_b\cap A|\le 1$ \quad and \quad $\mu(V_b)=\mu(V_b\cap A_{\mu})$}
\end{equation} 
for any $b\in \R^m$, where $V_b:=V+b$ and 
\begin{equation*}
A_{\mu}:=\{x\in \R^m\mid \mu(\{x\})>0\}.
\end{equation*}
\end{lem}
\begin{proof}
First, we define $A_0$ as
\begin{equation*}
A_0:=A\cup A_{\mu}.
\end{equation*}
For $k=1,\ldots, m$, we define $S_{k-1}$, $G_k$, and $A_k$ inductively as
\begin{align*}
S_{k-1}&:=\bigcup_{i=0}^{k-1}A_i,\\
G_k&:=\{W\in \Aff_k(\R^m) \mid \mu(W\setminus S_{k-1})>0\},\\
A_k&:=\bigcup_{W\in G_k}(W\setminus S_{k-1}),
\end{align*}
where $\Aff_k(\R^m)$ denotes the set of $k$-dimensional affine subspaces of $\R^m$. Then, we prove the following claim. 
\begin{claim}
\label{induc}
A proposition $(P_k)$ holds for any $k=0,1,\ldots,m-1$, and a proposition $(Q_k)$ holds for any $k=1,\ldots,m-1$, where
\begin{enumerate}
\item[$(P_k):$] For any $\ell=0,\ldots,k$ and $W\in \Aff_{\ell}(\R^m)$, we have
\begin{equation}
\label{nzero}
\mu(W\setminus S_k)=0.
\end{equation}
\item[$(Q_k):$] $G_k$ is countable, and $S_k$ is a Borel subset of $X$. 
\end{enumerate} 
\end{claim}
\begin{proof}
To prove Claim \ref{induc}, it is sufficient that we prove $(P_0)$, $(P_{k-1})\Rightarrow (Q_k)$, and $(Q_k)\Rightarrow (P_k)$ for any $k=1,\ldots,m-1$. We can easily prove $(P_0)$ since $S_0=A_0\supset A_{\mu}$.

First, we prove $(Q_k)\Rightarrow (P_k)$ for any $k=1,\ldots,m-1$. By $(Q_k)$, (\ref{nzero}) for $k$ is well-defined. In addition, we may assume $\ell=k$ since for any $W\in \Aff_{\ell}(\R^m)$, there exists $W'\in \Aff_k(\R^m)$ such that $W\subset W'$. We take any $W\in \Aff_k(\R^m)$. If $\mu(W\setminus S_{k-1})=0$, then $S_{k-1}\subset S_k$ implies (\ref{nzero}). If otherwise, then $W\setminus S_{k-1}\subset A_k$ by $W\in G_k$. Thus, we have 
\begin{align*}
W\subset S_{k-1}\cup (W\setminus S_{k-1})\subset S_{k-1}\cup A_k=S_k.
\end{align*}
Therefore, we get (\ref{nzero}), in particular, $(P_k)$ holds.

Next, we prove $(P_{k-1})\Rightarrow (Q_k)$ for any $k=1,\ldots,m-1$. In order to prove that $G_k$ is countable, it is sufficient to prove that 
\begin{equation*}
G_{k,\ell}:=\{W\in G_k\mid \mu(W\setminus S_{k-1})>1/\ell\}
\end{equation*}
consists of at most $\ell-1$ points for any $\ell\in \N$ since
\begin{equation*}
G_k=\bigcup_{\ell\in \N}G_{k,\ell}.
\end{equation*}
Suppose that there are distinct $W_1,\ldots,W_{\ell}\in G_{k,\ell}$. If $i\ne j$, then $W_i\cap W_j$ is empty, or $\dim (W_i\cap W_j)$ is less than $k$ by Remark \ref{affprop}. Hence, $(P_{k-1})$ implies 
\begin{equation*}
\mu((W_i\setminus S_{k-1})\cap(W_j\setminus S_{k-1}))=\mu((W_i\cap W_j)\setminus S_{k-1})=0.
\end{equation*}
Thus, we have
\begin{equation*}
\mu\left(\bigcup_{i=1}^{\ell}(W_i\setminus S_{k-1})\right)=\sum_{i=1}^{\ell}\mu\left(W_i\setminus S_{k-1}\right)>\sum_{i=1}^{\ell}\frac{1}{\ell}=1.
\end{equation*}
This contradicts that $\mu$ is a probability measure. Therefore, $G_k$ is countable. In addition, since $W\setminus S_{k-1}$ is Borel for any $W\in G_k$, $A_k$ and $S_k$ are Borel. In particular, $(Q_k)$ holds. 
\end{proof}
Here, since $A_0$ is countable, 
\begin{equation*}
G_1':=\{W\in \Aff_1(\R^m)\mid |W\cap A_0|\ge 2\}
\end{equation*}
and
\begin{equation*}
G:=\bigcup_{k=1}^{m-1}G_k\cup G_1'
\end{equation*}
are countable. For any $W\in G$, 
\begin{equation*}
\G_W:=\{V\in \Gr_{m-1}(\R^m)\mid \widetilde{W}\subset V\}
\end{equation*}
is a closed subset of the Grassmann manifold $\Gr_{m-1}(\R^m)$ and homeomorphic to the real projective space $\R P^{m-1-k}$, where $k$ satisfies $W\in G_k$ or $W\in G_k'$. Indeed, $\Gr_{m-1}(\R^m)$ is homeomorphic to $\R P^{m-1}$ by a map $\ph:\Gr_{m-1}(\R^m)\to \R P^{m-1}$ defined as 
\begin{equation*}
\ph(V):=[p]\qquad (V=p^{\perp}),
\end{equation*}
and we have 
\begin{equation*}
\ph(\G_W)=\left\{[p]\in \R P^{m-1}\ \left| \ p\in \widetilde{W}^{\perp}\right.\right\}.
\end{equation*} 
Since this set is equal to the image of an embedding $\tilde{\psi}:\R P^{m-1-k}\to \R P^{m-1}$ which is induced by an injective linear map $\psi:\R^{m-k}\xrightarrow{\cong} \widetilde{W}^{\perp}\hookrightarrow \R^m$, $\G_W$ is a compact (particularly closed) subset of $\Gr_{m-1}(\R^m)$ that is homeomorphic to $\R P^{m-1-k}$. Therefore, by 
\begin{equation*}
m-1-k<m-1=\dim \Gr_{m-1}(\R^m)
\end{equation*}
and countability of $G$, it follows from Lemma \ref{mfd} that there exists an element $V$ in
\begin{equation*}
\Gr_{m-1}(\R^m)\setminus\bigcup_{W\in G}\G_W.
\end{equation*}
If we take any $b\in \R^m$, then $V_b$ satisfies (\ref{cond}). 

Indeed, suppose $|V_b\cap A_0|\ge 2$, then there exists $W\in G_1'$ such that $W\subset V_b$. By the definition of $\widetilde{W}$, $\widetilde{W}\subset \widetilde{V_b}=V$ implies $V\in \G_W$. It is a contradiction. Thus, we have $|V_b\cap A_0|\le 1$. In particular, we obtain $|V_b\cap A|\le |V_b\cap A_0|\le 1$. Suppose that $\mu(V_b)>\mu(V_b\cap A_{\mu})$. Then, $\mu(V_b\setminus A_0)>0$ since $\mu(A\setminus A_{\mu})=0$. Hence, there exists $k_0\in \{1,\ldots,m-1\}$ such that
\begin{equation*}
\text{$\mu(V_b\setminus S_{k_0})=0$ \quad and \quad $\mu(V_b\setminus S_{k_0-1})>0$}
\end{equation*}
by (\ref{nzero}) and $S_0\subset S_1\subset\cdots\subset S_{m-1}$. Then, we have 
\begin{equation*}
\mu((V_b\setminus S_{k_0-1})\setminus A_{k_0})=0
\end{equation*}
since $S_{k_0-1}\cup A_{k_0}=S_{k_0}$.
On the other hand, for any $W\in G_{k_0}$, we have $V\notin \G_W$, i.e., $W\not\subset V_b$. By Remark \ref{affprop}, $V_b\cap W$ is empty, or $\dim (V_b\cap W)$ is less than $k_0$. Thus, we have $\mu((V_b\cap W)\setminus S_{k_0-1})=0$ by (\ref{nzero}). Therefore, we obtain
\begin{align*}
&\quad\ \mu((V_b\setminus S_{k_0-1})\cap A_{k_0})\\
&=\mu\left(\bigcup_{W\in G_{k_0}}((V_b\setminus S_{k_0-1})\cap(W\setminus S_{k_0-1}))\right)\\
&=\mu\left(\bigcup_{W\in G_{k_0}}((V_b\cap W)\setminus S_{k_0-1})\right)\\
&\le\sum_{W\in G_{k_0}}\mu((V_b\cap W)\setminus S_{k_0-1})\\
&=0
\end{align*}
and 
\begin{align*}
\mu(V_b\setminus S_{k_0-1})\le\mu((V_b\setminus S_{k_0-1})\cap A_{k_0})+\mu((V_b\setminus S_{k_0-1})\setminus A_{k_0})=0.
\end{align*}
However, this contradicts $\mu(V_b\setminus S_{k_0-1})>0$. Thus, we obtain $\mu(V_b)=\mu(V_b\cap A_{\mu})$. Therefore, $V_b$ satisfies (\ref{cond}).
\end{proof}
\begin{rem}
In the following, we put the countable set $A$ in Lemma \ref{afflem} as the set of $\mu$-atomic points, namely, $A:=A_{\mu}$ unless otherwise noted.
\end{rem}
\begin{cor}
\label{atom1Lip}
Let $m$ be a natural number and let $\mu$ be a Borel probability measure on $\R^m$. Then, there exists a $1$-Lipschitz function $f:(\R^m,\|\cdot\|_{\infty})\to \R$ such that 
\begin{itemize}
\item[(AP)] for any $b\in \R$ with $f_*\mu(\{b\})>0$, there exists $x\in f^{-1}(\{b\})$ satisfying $\mu(\{x\})=f_*\mu(\{b\})$.
\end{itemize}
\end{cor}
\begin{proof}
It follows from Lemma \ref{afflem} that there exists an $(m-1)$-dimensional linear subspace $V$ of $\R^m$ such that (\ref{cond}) holds for any $b\in \R^m$. Then, there exists $p\in \R^m$ with $\|p\|_1=1$ and $V=p^{\perp}$. We define a function $f:(\R^m,\|\cdot\|_{\infty})\to \R$ as $f(x):=\langle p,x\rangle$. Then, $f$ is a $1$-Lipschitz function on $(\R^m,\|\cdot\|_{\infty})$, and
\begin{equation*}
f^{-1}(\{a\})=V+a\cdot p
\end{equation*}
for any $a\in \R$. Therefore, Corollary \ref{atom1Lip} holds.
\end{proof}
\begin{proof}[Proof of Main Theorem \ref{atomultiobsdiam}]
We may assume $\|\bal\|_1\le 1$ since all of (1), (2), (3), and (4) do not hold if $\|\bal\|_1> 1$, and we can easily prove $(3)\Rightarrow(4)$ by Proposition \ref{multiobsdiamonotone}.

\begin{proof}[Proof of $(1)\Rightarrow (3)$]
From $(1)$, we can get an mm-space $Y$ dominating $X$ and distinct points $\{y_i\}_{i=1}^n\subset Y$ with $\mu_Y(\{y_i\})\ge \al_i$ for any $i\in \N_n$. Then, Lemma \ref{multidiamzero} implies $\diam(Y;\bal)=0$. Thus, we obtain $\udiam(X;\bal)\le \diam(Y;\bal)=0$, that is, $(3)$.
\end{proof}
\begin{proof}[Proof of $(3)\Rightarrow (2)$]
By $(3)$ and Remark \ref{diamrem}, we can take a sequence $\{(Y_m,F_m,\{A_{i,m}\}_{i=1}^n)\}_{m\in \N}$ in $\underline{D_X}(\bal)$ such that 
\begin{equation*}
\lim_{m\to\infty}\sup_{i\in \N_n}\diam A_{i,m}=0.
\end{equation*}
Then, we have 
\begin{equation*}
\diam \overline{F_m(A_{i,m})}=\diam F_m(A_{i,m})\le \diam A_{i,m}
\end{equation*}
for any $i\in \N_n$ and $m\in\N$. Hence, we obtain 
\begin{align*}
&\quad \ \mu_X\left(\bigcup_{i\in I}\overline{F_m(A_{i,m})}\right)=\mu_{Y_m}\left(\bigcup_{i\in I}F_m^{-1}\left(\overline{F_m(A_{i,m})}\right)\right)\\
&\ge \mu_{Y_m}\left(\bigcup_{i\in I}A_{i,m}\right)=\sum_{i\in I}\mu_{Y_m}(A_{i,m})\ge \sum_{i\in I}\al_i
\end{align*}
for any $I\subset \N_n$. Thus, Lemma \ref{convlem} for $\mu_m:=\mu_X$, $B_{i,m}:=\overline{F_m(A_{i,m})}$, and $\ep_m:=0$ implies $(2)$.
\end{proof}
\begin{proof}[Proof of $(2)\Rightarrow (1)$]
By $(2)$, we get $\{x_i\}_{i=1}^n\subset X$ such that 
\begin{equation*}
\mu_X(\{x_i\mid i\in I\})\ge \sum_{i\in I}\al_i 
\end{equation*}
for any $I\subset \N_n$. Since for any $A\in \B_X$, we have
\begin{align*}
\sum_{i=1}^n\al_i\de_{x_i}(A)=\sum_{i\in I_A}\al_i\le \mu_X(\{x_i\mid i\in I_A\})\le \mu_X(A),
\end{align*}
where $I_A:=\{i\in \N_n\mid x_i\in A\}$, a Borel measure $\nu_X$ defined as
\begin{equation*}
\nu_X:=\mu_X-\sum_{i=1}^n\al_i\de_{x_i}
\end{equation*}
is non-negative. Here, we define $Y$, $d_Y$, and $\mu_Y$ as
\begin{align*}
Y&:=\bigcup_{i=0}^n(Y_i\times \{i\}),\qquad Y_i:=
\begin{cases}
X\qquad &(i=0)\\
\{x_i\}\qquad &(i\in \N_n)
\end{cases}
,\\
d_Y&:Y\times Y\to [0,\infty)\\
 &d_Y((x,i),(y,j)):=d_X(x,y)+|i-j|,\\
\mu_Y&:=\nu_X\otimes \de_0+\sum_{i=1}^n\al_i\de_{(x_i,i)}.
\end{align*}
Then, $(Y,d_Y, \mu_Y)$ is an mm-space and $\{(x_i,i)\}_{i=1}^n$ are distinct points in $Y$ such that $\mu_Y(\{(x_i,i)\})\ge \al_i$ for any $i\in \N_n$. Furthermore, a map $F:Y\to X$ defined as $F(x,i):=x$ is $1$-Lipschitz and satisfies 
\begin{equation*}
F_*\mu_Y=\nu_X+\sum_{i=1}^n\al_i\de_{x_i}=\mu_X.
\end{equation*}
Thus, we have $Y\succ X$ and $(1)$.
\end{proof}
\begin{proof}[Proof of $(4)\Rightarrow (2)$] 
Suppose $\uObsdiam(X;\bal)=0$. By Proposition \ref{findimappro}, for any $m\in \N$, there exists $\umu_m\in \P(\R^m)$ such that we have
\begin{align*}
X_1\prec X_2\prec\cdots\prec X_m\prec\cdots \prec X\text{\qquad and \qquad} X_m\Boxto X\qquad(m\to\infty),
\end{align*}
where $X_m:=(\R^m,\|\cdot\|_{\infty},\umu_m)$. Then, it follows from Corollary \ref{atom1Lip} that for any $m\in \N$, there exists a $1$-Lipschitz function $f_m\in \Lip_1(X_m)$ satisfying (AP). For any $m\in \N$, we have
\begin{equation*}
\udiam((\R,(f_m)_*\umu_m);\bal)\le \uObsdiam(X_m;\bal)\le \uObsdiam(X;\bal)=0.
\end{equation*}
Hence, $(3)\Rightarrow(2)$ implies that there exists $\{b_{i,m}\}_{i=1}^n\subset \R$ such that 
\begin{equation*}
(f_m)_*\umu_m(\{b_{i,m}\mid i\in I\})\ge \sum_{i\in I}\al_i
\end{equation*}
for any $I\subset \N_n$ and $m\in \N$. By (AP), there exists $\{x_{i,m}\}_{i=1}^n\subset X_m$ such that $f_m(x_{i,m})=b_{i,m}$ and $\umu_m(\{x_{i,m}\})=(f_m)_*\umu_m(\{b_{i,m}\})$ for any $i\in \N_n$ and $m\in \N$. Then, we have
\begin{equation*}
\umu_m(\{x_{i,m}\mid i\in I\})=(f_m)_*\umu_m(\{b_{i,m}\mid i\in I\})\ge \sum_{i\in I}\al_i.
\end{equation*}
Furthermore, it follows from Proposition \ref{epmm} that there exist $\ep_m\searrow 0$ and an $\ep_m$-mm-isomorphism $p_m:X_m\to X$ by $X_m\Boxto X$. Thus, Lemma \ref{convlem} implies $(2)$ as $\mu_m:=(p_m)_*\umu_m$ and $B_{i,m}:=\{p_m(x_{i,m})\}$.
\end{proof}
This completes the proof of Main Theorem \ref{atomultiobsdiam}.
\end{proof} 
\section{Properties of the multivariable diameters}
\label{Propmultidiam}

In this section, we give some results for the multivariable partial (observable) diameter.

\begin{prop}
Let $X$ be an mm-space and let $\bal\in \A$. Then, we have
\begin{enumerate}
\item $\diam(X;\|\bal\|_{\infty})\le \udiam(X;\bal)\le \diam(X;\|\bal\|_1)$,
\item $\Obsdiam(X;\|\bal\|_{\infty})\le \uObsdiam(X;\bal)\le \Obsdiam(X;\|\bal\|_1)$.
\end{enumerate}
In particular, for any sequence $\{X_m\}_{m\in \N}$ of mm-spaces, 
\begin{equation*}
\lim_{m\to\infty}\uObsdiam(X_m;\bal)=0
\end{equation*}
for any $\bal\in \A$ with $\|\bal\|_1<1$ if and only if $\{X_m\}_{m\in \N}$ is a L\'{e}vy family.
\end{prop}
\begin{proof}
Since it is easy to prove $(2)$ from $(1)$, we prove only $(1)$. 

First, we prove $\diam(X;\|\bal\|_{\infty})\le \udiam(X;\bal)$. If we take any $\{\mu_i\}_{i=1}^n\in \M_X(\bal)$ and set $i_0$ as $\al_{i_0}=\|\bal\|_{\infty}$, then 
\begin{equation*}
\mu_X(\supp \mu_{i_0})\ge \sum_{i=1}^n\al_i\mu_i(\supp \mu_{i_0})\ge \al_{i_0}=\|\bal\|_{\infty}
\end{equation*}
implies 
\begin{equation*}
\diam(X;\|\bal\|_{\infty})\le \diam \supp \mu_{i_0}\le \sup_{i\in \N_n}\diam \supp \mu_i.
\end{equation*}
Thus, we obtain $\diam(X;\|\bal\|_{\infty})\le \udiam(X;\bal)$.

Next,  we prove $\udiam(X;\bal)\le \diam(X;\|\bal\|_1)$. We take any $A\in \B_X$ with $\mu_X(A)\ge \|\bal\|_1$. We define probability measures $\{\mu_i\}_{i=1}^n$ as 
\begin{align*}
\mu_i:=(\mu_X)_A=\frac{\mu_X(\cdot\cap A)}{\mu_X(A)}
\end{align*}
for any $i\in \N_n$. Then, 
\begin{align*}
\mu_X\ge \mu_X(\cdot \cap A)\ge \sum_{i=1}^n\frac{\al_i}{\mu_X(A)}\mu_X(\cdot \cap A)=\sum_{i=1}^n\al_i\mu_i
\end{align*}
implies $\{\mu_i\}_{i=1}^n\in \M_X(\bal)$. In addition, $\supp \mu_i$ is a subset of $\overline{A}$ due to $\mu_i(\overline{A})=1$ for any $i\in \N_n$. Therefore, we get 
\begin{equation*}
\udiam(X;\bal)\le \sup_{i\in \N_n}\diam \supp \mu_i\le \diam \overline{A}=\diam A.
\end{equation*}
Thus, we obtain $\udiam(X;\bal)\le \diam(X;\|\bal\|_1)$.
\end{proof}
Next, we prove the lower-semicontinuity of the multivariable partial and observable diameters.
\begin{prop}
\label{partlsc}
Let $\{X_m\}_{m\in \N}$ be a sequence of mm-spaces and let $\bal\in \A$. If $X_m$ $\Box$-converges to $X\in \X$, then we have
\begin{equation*}
\udiam(X;\bal)\le \liminf_{m\to \infty}\udiam(X_m;\bal).
\end{equation*}
In other words, $\bal$-partial diameter is lower-semicontinuous with respect to the box topology.
\end{prop}
\begin{proof}
We set $R:=\liminf_{m\to \infty}\udiam(X_m;\bal)$. By taking a subsequence, we may assume that $\udiam(X_m;\bal)$ converges to $R$. Then, for any $m\in \N$, there exists $\{\mu_{i,m}\}_{i=1}^n\in \M_{X_m}(\bal)$ such that 
\begin{equation*}
\lim_{m\to\infty}\sup_{i\in \N_n}\diam \supp \mu_{i,m}=R.
\end{equation*}
Furthermore, it follows from Proposition \ref{epmm} that there exist an $\ep_m$-mm-isomorphism $p_m:X_m\to X$ and a Borel subset $\tX_m\subset X_m$ such that $\ep_m\searrow 0$. Then,  a sequence $\{(p_m)_*\mu_{i,m}\}_{m\in \N}$ is tight for any $i\in \N_n$ because of the tightness of $\{(p_m)_*\mu_{X_m}\}_{m\in \N}$ and the definition of $\M_{X_m}(\bal)$. Hence, $\{(p_m)_*\mu_{i,m}\}_{m\in \N}$ has a weakly convergent subsequence by Prokhorov's theorem. Moreover, a sequence $\{A_{i,m}\}_{m\in \N}$ of closed subsets of $X$ defined as 
\begin{equation*}
A_{i,m}:=\overline{p_m(\tX_m\cap \supp \mu_{i,m})}
\end{equation*}
also has a weakly convergent subsequence due to Proposition \ref{wHsub}. By diagonal argument, there exist a subsequence $\{m_k\}_{k\in \N}$, closed subsets $\{A_i\}_{i=1}^n$ of $X$, and $\{\mu_i\}_{i=1}^n\subset \P(X)$ such that $(p_{m_k})_*\mu_{i,{m_k}}$ converges weakly to $\mu_i$ and $A_{i,m_k}$ converges weakly to $A_i$ as $k\to\infty$ for any $i\in \N_n$. Then, $\{\mu_i\}_{i=1}^n$ is an element in $\M_X(\bal)$, namely, the following inequality holds,
\begin{equation}
\label{measine}
\sum_{i=1}^n\al_i\mu_i\le \mu_X.
\end{equation}
\begin{proof}[Proof of $(\ref{measine})$]
We take any $\ep>0$ and $\ell\in \N$ with $\ell\le n$. Then, there exists a sufficiently large $k\in \N$ such that for any $i\in \N_\ell$, we have
\begin{align*}
d_P((p_{m_k})_*\mu_{i,{m_k}},\mu_i)<\ep,\\
d_P((p_{m_k})_*\mu_{X_{m_k}}),\mu_X)<\ep.
\end{align*}
For any closed subset $K\subset X$, 
\begin{align*}
\sum_{i=1}^{\ell}\al_i\mu_i(K)&\le \sum_{i=1}^{\ell}\al_i((p_{m_k})_*\mu_{i,{m_k}}(N_{\ep}(K))+\ep)\\
&=(p_{m_k})_*\left(\sum_{i=1}^{\ell}\al_i\mu_{i,{m_k}}\right)(N_{\ep}(K))+\ep\sum_{i=1}^{\ell}\al_i\\
&\le(p_{m_k})_*\mu_{X_{m_k}}(N_{\ep}(K))+\|\bal\|_1\ep\\
&=\mu_X(N_{2\ep}(K))+(1+\|\bal\|_1)\ep.
\end{align*}
By taking as $\ep\to 0$, we obtain 
\begin{equation*}
\sum_{i=1}^{\ell}\al_i\mu_i(K)\le \mu_X(K).
\end{equation*}
Next, for any Borel subset $B\subset X$, there exists a sequence $\{K_N\}_{N\in \N}$ of compact subsets of $X$ such that we have  
\begin{align*}
K_1\subset K_2\subset \cdots\subset K_N\subset \cdots\subset B,\\
\mu_i\left(B\setminus \bigcup_{N\in \N}K_N\right)=0\qquad(\forall i\in \N_{\ell})
\end{align*}
since $\mu_1,\ldots,\mu_{\ell}$ are tight. Therefore, we get 
\begin{align*}
\sum_{i=1}^{\ell}\al_i\mu_i(B)&=\sum_{i=1}^{\ell}\al_i\mu_i\left(B\setminus \bigcup_{N\in \N}K_N\right)+\lim_{N\to \infty}\sum_{i=1}^{\ell}\al_i\mu_i(K_N)\\
&\le 0+\lim_{N\to \infty}\mu_X(K_N)\le \mu_X(B).
\end{align*}
Thus, we obtain $(\ref{measine})$ by taking as $\ell\nearrow n$.
\end{proof} 
Moreover, Proposition \ref{wHconti} implies
\begin{align*}
\mu_i(A_i)&\ge \limsup_{k\to \infty}(p_{m_k})_*\mu_{i,{m_k}}(A_{i,m_k})\\
&\ge \limsup_{k\to \infty}\mu_{i,m_k}(\tX_{m_k}\cap \supp \mu_{i,{m_k}})=1,
\end{align*}
in particular, we have $\supp \mu_i\subset A_i$ for any $i\in \N_n$. Therefore, we obtain
\begin{align*}
\sup_{i\in\N_n}\diam A_i&\le \sup_{i\in\N_n}\liminf_{k\to\infty}\diam A_{i,m_k}\\
&\le \liminf_{k\to\infty}(\sup_{i\in\N_n}\diam \supp \mu_{i,{m_k}}+2\ep_{m_k})\le R.
\end{align*}
This completes the proof.
\end{proof}
\begin{prop}
\label{Obslsc}
Let $\{X_m\}_{m\in \N}$ be a sequence of mm-spaces and let $\bal\in \A$. If $X_m$ $\Box$-converges to $X\in \X$, then we have
\begin{equation*}
\uObsdiam(X;\bal)\le \liminf_{m\to \infty}\uObsdiam(X_m;\bal).
\end{equation*}
In other words, $\bal$-observable diameter is lower-semicontinuous with respect to the box topology.
\end{prop}
\begin{proof}
We take any $f\in \Lip_1(X)$. From $X_m\Boxto X$ and Proposition \ref{epmm}, there exists an $\ep_m$-mm-isomorphism $p_m:X_m\to X$ for $\ep_m\searrow 0$. Since $f\circ p_m$ satisfies the assumption of Proposition \ref{Lipappro}, there exists $f_m\in \Lip_1(X_m)$ such that $\dKF(f\circ p_m,f_m)\le \ep_m$. Propositions \ref{Prolem1} and \ref{Prolem2} imply 
\begin{align*}
&\quad d_P((f_m)_*\mu_{X_m}, f_*\mu_X)\\
&\le d_P((f_m)_*\mu_{X_m},(f\circ p_m)_*\mu_{X_m})+d_P((f\circ p_m)_*\mu_{X_m},f_*\mu_X)\\
&\le \dKF(f_m,f\circ p_m)+d_P((p_m)_*\mu_{X_m},\mu_X)\\
&\le 2\ep_m,
\end{align*}
in particular, $(f_m)_*\mu_{X_m}$ converges weakly to $f_*\mu_X$. Therefore, Propositions \ref{Boxwconv} and \ref{partlsc} imply 
\begin{align*}
\udiam((\R,f_*\mu_X);\bal)&\le \liminf_{m\to\infty}\udiam((\R,(f_m)_*\mu_{X_m});\bal)\\
&\le \liminf_{m\to \infty}\uObsdiam(X_m;\bal).
\end{align*}
From the arbitrariness of $f$, we obtain the lower-semicontinuity of the $\bal$-observable diameter. 
\end{proof}
\begin{cor}[cf. \cite{Ozawa-Shioya}]
$\bal$-partial diameter and $\bal$-observable diameter are lower-semicontinuous with respect to the concentration topology.
\end{cor}
This corollary follows from Proposition \ref{multiobsdiamonotone} and {\cite[Lemma 3.1]{Esaki-Kazukawa-Mitsuishi}}.

\section{Another definition of the multivariable partial and observable diameter}
\label{andef}

In this section, we give another definition of the multivariable partial and observable diameters and prove an analog of Main Theorem \ref{atomultiobsdiam} for these multivariable partial and observable diameters.
\begin{dfn}
Let $X$ be an mm-space and let $\bal\in \A$. We define the \emph{$\bal$-partial diameter $\diam'(X;\bal)$, $\diam''(X;\bal)$ of $X$}, and the \emph{$\bal$-observable diameter $\Obsdiam''(X;\bal)$ of $X$} by
\begin{align*}
\diam'(X;\bal)&:=
\inf\left\{\left.\sup_{i\in \N_n}\diam A_i\ \right| \ \{A_i\}_{i=1}^n\in D_X(\bal)\right\},\\
\diam''(X;\bal)&:=\sup_{\bal'\le \bal}\diam'(X;\bal'),\\
\Obsdiam''(X;\bal)&:=\sup\{\diam''((\R, f_*\mu_X);\bal)\mid f\in \Lip_1(X)\},
\end{align*}
where $\diam'(X;\bal)=0$ if $D_X(\bal)=\emptyset$, and $\bal'\le \bal$ means $n(\bal')=n(\bal)=n$ and $\al_i'\le \al_i$ for any $i\in \N_n$. 
\end{dfn}
\begin{lem}
\label{X}
Let $X$ be an mm-space and let $\bal\in \A$. If $\mu_X$ has a non-atomic part, namely, 
\begin{equation*}
\sum_{x\in X}\mu_X(\{x\})<1,
\end{equation*}
then there exists $\bal'\in \A$ such that $\bal'\le \bal$ and $D_X(\bal')\ne \emptyset$.
\end{lem}
\begin{proof}
We define $\nu_X$ as
\begin{equation*}
\nu_X:=\mu_X-\sum_{x\in X}\mu_X(\{x\})\de_x.
\end{equation*}
By the assumption, $\nu_X$ is a non-zero Borel measure on $X$. Then, there exists a Borel isomorphism $f:(X,\nu_X)\to ([0,\nu_X(X)],\LL^1|_{[0,\nu_X(X)]})$ satisfying $f_*\nu_X=\LL^1|_{[0,\nu_X(X)]}$ by \cite{Kechris}*{(17.41)}. Here, for any $i\in\N_n$, we set $\al_i':=\min\{\al_i,2^{-i}\nu_X(X)\}>0$ and $A_i:=f^{-1}([b_{i-1},b_i))$ where $b_0:=0$, $b_i:=\sum_{j=1}^{i}\al_j'$. Then, we obtain $\bal'\le \bal$ and $\{A_i\}_{i=1}^n\in D_X(\bal')$. 
\end{proof}
\begin{prop}
Let $X$ be an mm-space and let $\bal\in \A$. Then, $\diam''(X;\bal)=0$ is equivalent to either one of the following $(1)$ and $(2)$.
\begin{enumerate}
\item There are distinct points $\{x_i\}_{i=1}^n\subset X$ such that $\mu_X(\{x_i\})\ge \al_i$ for any $i\in \N_n$.
\item $D_X(\bal)=\emptyset$ and $\supp\mu_X$ consists of at most $n=n(\bal)$ points.
\end{enumerate}
\end{prop}
\begin{proof}
It is easy to see that $\diam''(X;\bal)=0$ follows from either of $(1)$ and $(2)$. Therefore, we prove only that $\diam''(X;\bal)=0$ implies $(1)$ or $(2)$. In other words, it is sufficient to prove that $\diam''(X;\bal)=0$ and negation of $(1)$ imply $(2)$. Furthermore, we may assume that $\bal$ is non-increasing.

First, we see that $\diam''(X;\bal)=0$ and negation of $(1)$ imply $D_X(\bal)=\emptyset$ from Lemma \ref{multidiamzero}. 

Second, we prove that $\diam''(X;\bal)=0$ and negation of $(1)$ imply $|\supp\mu_X|\le n$. We set $A:=\{x\in X\mid \mu_X(\{x\})>0\}$ and 
\begin{equation*}
\nu_X:=\mu_X-\sum_{x\in A}\mu_X(\{x\})\de_x.
\end{equation*}
 If $|A|< n$ and $\nu_X(X)=0$, then we can get easily $|\supp\mu_X|=|A|< n$. If $|A|< n$ and $\nu_X(X)>0$, then $\diam''(X;\bal)=0$ implies that $D_X(\bal')=\emptyset$ for any $\bal'\le \bal$ by Lemma \ref{multidiamzero}. This contradicts Lemma \ref{X}. Thus, we may assume $|A|\ge n$. We take distinct points $\{y_i\}_{i=1}^n\subset A$ with $\mu_X(\{y_1\})\ge \mu_X(\{y_2\})\ge \cdots$ and $\mu(\{y_i\})\ge \mu(\{y\})$ for any $i\in \N_n$ and any $y\in X\setminus \{y_i\mid i\in \N_n\}$. Then, it follows from $D_X(\bal)=\emptyset$ that there exists $i\in \N_n$ such that $\mu_X(\{y_i\})<\al_i$, and we set 
\begin{equation*}
i_0:=\min\{i\in \N_n\mid \mu_X(\{y_i\})<\al_i\}.
\end{equation*}
If $|\supp\mu_X|>n$, then $B:=X\setminus \{y_i\mid i\in \N_n\}$ has positive measure. Here, if we set 
\begin{equation*}
(\al'_i, A_i):=
\begin{cases}
(\min\{\al_i,\mu_X(\{y_i\})\}, \{y_i\})\qquad&(i\ne i_0)\\
(\min\{\al_i,\mu_X(B\cup\{y_i\})\}, B\cup\{y_i\})\qquad&(i= i_0)
\end{cases}
,
\end{equation*}
then we have $\bal'\le \bal$ and $\{A_i\}_{i=1}^n\in D_X(\bal')$. Thus, by $\diam''(X;\bal)=0$, there are distinct points $\{x_i\}_{i=1}^n\subset X$ such that $\mu_X(\{x_i\})\ge \al_i'$ for any $i\in \N_n$. For any $i<i_0$, we get $\mu_X(\{x_i\})\ge \al_i'=\al_i\ge \al_{i_0}\ge \al_{i_0}'$ and $\mu_X(\{x_{i_0}\})\ge \al_{i_0}'$. However, the definitions of $\{y_i\}_{i=1}^n$ and $i_0$ imply 
\begin{equation*}
\{x\in X\mid \mu_X(\{x\})\ge \al_{i_0}'\}=\{y_1,\ldots,y_{i_0-1}\}.
\end{equation*}
It is a contradiction. Therefore, we have $|\supp\mu_X|\le n$, i.e., (2) holds.
\end{proof}
\begin{prop}[cf. Main Theorem \ref{atomultiobsdiam}]
\label{prop}
Let $X$ be an mm-space and let $\bal\in \A$ with $n=n(\bal)<\infty$. If $\Obsdiam''(X;\bal)=0$, then we have either $D_X(\bal)=\emptyset$ or there are distinct points $\{x_i\}_{i=1}^n\subset X$ such that $\mu_X(\{x_i\})\ge \al_i$ for any $i\in \N_n$.
\end{prop}
To prove this proposition, we prepare several lemmas.
\begin{lem}
\label{lemA}
Let $\{(X_m,\mu_m)\}_{m\in \N}$ be a sequence of probability measure spaces and let $\{p_{m,k}:X_k\to X_m\}_{m\le k}$ be a family of measurable maps satisfying the following $(1)$, $(2)$, and $(3)$.
\begin{enumerate}
\item For any $m\in \N$, $p_{m,m}=\id_{X_m}$.
\item For any $m\le k\le \ell$, $p_{m,\ell}=p_{m,k}\circ p_{k,\ell}$.
\item For any $m\le k$, $(p_{m,k})_*\mu_k=\mu_m$.
\end{enumerate}
Suppose that there exist $\al>0$ and $\{x_m\}_{m\in \N}\in \prod_{m\in \N}X_m$ such that $\mu_m(\{x_m\})\ge \al$ for any $m\in \N$, then there exists a subsequence $\{m_k\}_{k\in \N}$ such that $p_{m_k,m_{k+1}}(x_{m_{k+1}})=x_{m_k}$ for any $k\in \N$.
\end{lem}
\begin{proof}
We firstly prove the next claim.
\begin{claim}
\label{claim L}
For any infinity set $A\subset \N$, there exists $m\in A$ such that $\{k\in A\mid \text{$k>m$ and $p_{m,k}(x_k)=x_m$}\}$ is infinite.
\end{claim} 
\begin{proof}[Proof of Claim \ref{claim L}]
Suppose that there exists an infinite set $A\subset \N$ such that $\{k\in A\mid \text{$k>m$ and $p_{m,k}(x_k)=x_m$}\}$ is finite for any $m\in A$. 

If we set $k_1:=\min A$, then $\{p_{k_1,k}(x_k)\mid \text{$k\in A$ with $k>k_1$}\}$ is a finite subset of $X_{k_1}$ since $\mu_{k_1}(\{p_{k_1,k}(x_k)\})\ge \mu_k(\{x_k\})\ge \al$ for any $k\in A$ with $k>k_1$. Hence, there exists $y_1\in X_{k_1}$ such that 
\begin{equation*}
A_1:=\{k\in A\mid \text{$k>k_1$ and $p_{k_1,k}(x_k)=y_1$}\}
\end{equation*}
is infinite. By the assumption, we have $y_1\ne x_{k_1}$ and set $k_2:=\min A_1$. 

Next, $\{p_{k_2,k}(x_k)\mid \text{$k\in A_1$ with $k>k_2$}\}$ is a finite subset of $X_{k_2}$ since $\mu_{k_2}(\{p_{k_2,k}(x_k)\})\ge \mu_k(\{x_k\})\ge \al$ for any $k\in A_1$ with $k>k_2$. Hence, we see that there exists $y_2\in X_{k_2}$ such that 
\begin{equation*}
A_2:=\{k\in A_1\mid \text{$k>k_2$ and $p_{k_2,k}(x_k)=y_2$}\}
\end{equation*}
is infinite. By the assumption, we have $y_2\ne x_{k_2}$ and set $k_3:=\min A_2$.

Repeating this argument, we can find finite sequences 
\begin{equation*}
\text{$k_1<k_2<k_3<\cdots<k_N$ \quad and \quad$\{y_i\}_{i=1}^{N-1}\in \prod_{i=1}^{N-1}X_{k_i}$},
\end{equation*}
where $N\in \N$ satisfies $\al N>1$. Then, we obtain $\mu_{k_{N-\ell}}(\{y_{N-\ell}\})\ge \ell\al$ for any $\ell=1,\ldots,N-1$. Indeed, we have 
\begin{equation*}
\mu_{k_{N-1}}(\{y_{N-1}\})=\mu_{k_{N-1}}(\{p_{k_{N-1},k_N}(x_{k_N})\})\ge \al.
\end{equation*}
In addition, if $\mu_{k_{N-\ell}}(\{y_{N-\ell}\})\ge \ell\al$, then it follows from $x_{k_{N-\ell}}\ne y_{N-\ell}$ and 
\begin{equation*}
p_{k_{N-\ell-1},k_{N-\ell}}(x_{k_{N-\ell}})=p_{k_{N-\ell-1},k_{N-\ell}}(y_{N-\ell})=y_{N-\ell-1}
\end{equation*}
that
\begin{align*}
\mu_{k_{N-\ell-1}}(\{y_{N-\ell-1}\})&\ge \mu_{k_{N-\ell}}(\{x_{k_{N-\ell}}\})+\mu_{k_{N-\ell}}(\{y_{N-\ell}\})\\
&\ge \al+\ell\al=(\ell+1)\al.
\end{align*}
Thus, we obtain $\mu_{k_1}(\{y_1\})\ge (N-1)\al$ and 
\begin{equation*}
\mu_{k_1}(\{x_{k_1},y_1\})\ge \al+(N-1)\al=N\al>1.
\end{equation*}
This contradicts that $\mu_{k_1}$ is a probability measure. Therefore, Claim \ref{claim L} holds.
\end{proof}
For $A=\N$, Claim \ref{claim L} implies that there exists $m_1\in \N$ such that $B_1:=\{k\in \N\mid \text{$k>m_1$ and $p_{m_1,k}(x_k)=x_{m_1}$}\}$ is infinite. Next, for $A=B_1$, Claim \ref{claim L} also implies that there exists $m_2\in B_1$ such that $B_2:=\{k\in B_1\mid \text{$k>m_2$ and $p_{m_2,k}(x_k)=x_{m_2}$}\}$ is infinite. A sequence $\{m_k\}_{k\in \N}$ obtained by repeatedly applying Claim \ref{claim L} in this way satisfies $p_{m_{k+1},m_k}(x_{m_{k+1}})=x_{m_k}$ for any $k\in \N$. This completes the proof.
\end{proof}
\begin{lem}
\label{lemB}
Let $m$ be a natural number and let $\mu$ be a Borel probability measure on $\R^m$. If a $1$-Lipschitz function $f:X\to \R$ satisfies $(AP)$ of Corollary \ref{atom1Lip} and $D_X(\bal)\ne \emptyset$ where $X:=(\R^m,\|\cdot\|_{\infty}, \mu)$, then we have $D_{(\R,f_*\mu)}(\bal)\ne \emptyset$.
\end{lem}
\begin{proof}
We take $\{A_i\}_{i=1}^n\in D_X(\bal)$ and set $A:=\{x\in X\mid \mu(\{x\})>0\}$ and $A':=\{b\in \R\mid f_*\mu(\{b\})>0\}$. Then, it follows from (AP) of Corollary \ref{atom1Lip} that $f|_A:A\to A'$ is bijective and measure-preserving. If we define a finite measure $\nu$ on $\R$ as
\begin{equation*}
\nu:=f_*\mu-\sum_{b\in A'}f_*\mu(\{b\})\de_b,
\end{equation*}
then $\nu$ satisfies $\nu(\R)=1-f_*\mu(A')=1-\mu(A)$, and $\nu((-\infty, \cdot))$ is a non-decreasing and continuous function. We see that there exists a non-decreasing sequence $\{t_k\}_{k=1}^n\subset [-\infty,\infty]$ such that for any $k\in \N_n$,
\begin{equation*}
\nu((-\infty, t_k))=\mu\left(\bigcup_{i=1}^k(A_i\setminus A)\right).
\end{equation*}
If we define $A_i':=((t_{i-1},t_i)\setminus A')\cup f(A_i\cap A)$ for any $i\in \N_n$ where $t_0:=-\infty$, then $\{A_i'\}_{i=1}^n\in D_{(\R,f_*\mu)}(\bal)$. Indeed, $A_i'$ is Borel since $A$ is finite, and $\{A_i'\}_{i=1}^n$ are disjoint since $\{A_i\}_{i=1}^n$ and $\{(t_{i-1},t_i)\}_{i=1}^n$ are disjoint. In addition, we have 
\begin{align*}
f_*\mu(A_i')&=f_*\mu(((t_{i-1},t_i)\setminus A')\cup f(A_i\cap A))\\
&=f_*\mu((t_{i-1},t_i)\setminus A')+f_*\mu(f(A_i\cap A))\\
&=\nu((t_{i-1},t_i))+\mu(A_i\cap A)\\
&=\mu\left(\bigcup_{k=1}^i(A_k\setminus A)\right)-\mu\left(\bigcup_{k=1}^{i-1}(A_k\setminus A)\right)+\mu(A_i\cap A)\\
&=\mu\left(A_i\setminus A\right)+\mu(A_i\cap A)\\
&=\mu(A_i)\ge \al_i
\end{align*}
for any $i\in \N_n$. This completes the proof.
\end{proof}
\begin{lem}
\label{lemC}
Let $X$ be an mm-space and let $\bal\in \A$ with $n=n(\bal)<\infty$ and $D_X(\bal)\ne \emptyset$. For any $m\in \N$ and $\ep_m\in (0,\min\{\al_1,\ldots,\al_n\})$, there exists $\mu_m\in \P(\R^{n+m})$ such that we have
\begin{enumerate}
\item $Y_1\prec Y_2\prec\cdots\prec Y_m\prec\cdots \prec X$,\item $Y_m\Boxto X$ \qquad $(m\to\infty)$,
\item $D_{Y_m}(\bal-\bep_m)\ne \emptyset$,
\end{enumerate}
where $Y_m:=(\R^{n+m},\|\cdot\|_{\infty},\mu_m)$ and $\bal-\bep_m:=(\al_1-\ep_m,\ldots,\al_n-\ep_m)$.
\end{lem}
\begin{proof}
We take any $m\in \N$, $\ep_m\in (0,\min\{\al_1,\ldots,\al_n\})$, and $\{A_i\}_{i=1}^n\in D_X(\bal)$. By Proposition \ref{findimappro}, there exists $\umu_m\in \P(\R^m)$ such that
\begin{align*}
X_1\prec X_2\prec\cdots\prec X_m\prec\cdots \prec X\text{\qquad and \qquad} X_m\Boxto X\qquad(m\to\infty),
\end{align*}
where $X_m:=(\R^m,\|\cdot\|_{\infty},\umu_m)$. Then, from the construction of $\umu_m$ in \cite{ShioyaMMG}, for any $m\in \N$, there exists a dominating map $f_m:X\to X_m$ such that $f_m=p_m\circ f_{m+1}$ where $p_m:X_{m+1}\to X_m$ is the projection defined as $p_m(x_1,\ldots,x_m,x_{m+1}):=(x_1,\ldots,x_m)$. We define a $1$-Lipschitz map $g_m:X\to (\R^{n+m},\|\cdot\|_{\infty})$ as
\begin{equation*}
g_m(x):=(d_X(x,K_1),d_X(x,K_2),\ldots, d_X(x,K_n), f_m(x)),
\end{equation*}
where $\{K_i\}_{i=1}^n$ is a family of compact subsets of $X$ satisfying $K_i\subset A_i$ and $\mu_X(K_i)\ge \al_i-\ep_m$. If we set $\mu_m:=(g_m)_*\mu_X$, then we have (1),(2), and (3). Indeed, we get $Y_m\prec X$ and $Y_m\prec Y_{m+1}$ by dominating maps $g_m$ and $p_{n+m}$ respectively. Furthermore, we have $X_m\prec Y_m$ by a dominating map $p_m':\R^{n+m}\to \R^m$ defined as $p_m'(x_1,\ldots,x_{n+m}):=(x_{n+1},\ldots, x_{n+m})$. Then, $X_m\prec Y_m\prec X$ and $X_m\Boxto X$ imply $Y_m\Boxto X$ by Proposition \ref{mmsqueeze}. Finally, we have $\{g_m(K_i)\}_{i=1}^n\in D_{Y_m}(\bal-\bep_m)$ due to 
\begin{equation*}
\mu_m(g_m(K_i))\ge \mu_X(K_i)\ge \al_i-\ep_m,
\end{equation*}
and
\begin{equation*}
\min_{i\ne j}d_{Y_m}(g_m(K_i), g_m(K_j))\ge \min_{i\ne j}d_X(K_i,K_j)>0.
\end{equation*}
This completes the proof.
\end{proof}
\begin{proof}[Proof of Proposition \ref{prop}]
To prove Proposition \ref{prop}, it is sufficient that $\Obsdiam''(X;\bal)=0$ and $D_X(\bal)\ne\emptyset$ imply that there are distinct points $\{x_i\}_{i=1}^n\subset X$ such that $\mu_X(\{x_i\})\ge \al_i$ for any $i\in \N_n$. 
By Lemma \ref{lemC}, for any $m\in \N$ and 
\begin{equation*}
\ep_m:=\frac{1}{2m}\min\{\al_1,\ldots,\al_n\}\in (0,\min\{\al_1,\ldots,\al_n\}),
\end{equation*}
there exists $\mu_m\in \P(\R^{n+m})$ such that we have
\begin{enumerate}
\item $Y_1\prec Y_2\prec\cdots\prec Y_m\prec\cdots \prec X$,\item $Y_m\Boxto X$ \qquad $(m\to\infty)$,
\item $D_{Y_m}(\bal-\bep_m)\ne \emptyset$,
\end{enumerate}
where $Y_m:=(\R^{n+m},\|\cdot\|_{\infty},\mu_m)$. In addition, if we take a $1$-Lipschitz function $f_m:Y_m\to \R$ satisfying (AP) of Corollary \ref{atom1Lip}, then Lemma \ref{lemB} implies $D_{(\R,(f_m)_*\mu_m)}(\bal-\bep_m)\ne \emptyset$. By this and 
\begin{equation*}
\diam'((\R,(f_m)_*\mu_m); \bal-\bep_m)\le \Obsdiam''(X;\bal)=0,
\end{equation*}
Lemma \ref{multidiamzero} implies that there are distinct points $\{b_{i,m}\}_{i=1}^n\subset \R$ such that $(f_m)_*\mu_m(\{b_{i,m}\})\ge \al_i-\ep_m$ for any $i\in \N_n$ and $m\in \N$. Moreover, it follows from (AP) of Corollary \ref{atom1Lip} that there are distinct points $\{x_{i,m}\}_{i=1}^n\subset Y_m$ such that $\mu_m(\{x_{i,m}\})=(f_m)_*\mu_m(\{b_{i,m}\})\ge \al_i-\ep_m$ for any $i\in \N_n$ and $m\in \N$. Then, we have $\ep_m\searrow 0$ and 
\begin{equation*}
\min\{\al_1-\ep_m,\ldots,\al_n-\ep_m\}\ge \left(1-\frac{1}{2m}\right)\min\{\al_1,\ldots,\al_n\}\ge \ep_1=:\al.
\end{equation*}
By Lemma \ref{lemA} for $\{p_{m,k}:Y_k\to Y_m\}_{m\le k}$ defined as 
\begin{equation*}
p_{m,k}:=p_{n+m}\circ p_{n+m+1}\circ\cdots\circ p_{n+k-1},
\end{equation*} 
where $p_m$ is the projection defined in the proof of Lemma \ref{lemC} and diagonal argument, there exists a subsequence $\{m_k\}_{k\in \N}$ such that we have $p_{m_{k+1},m_k}(x_{i,m_{k+1}})=x_{i,m_k}$ for any $i\in \N_n$ and $k\in \N$.

Then, for an $\ep_k'$-mm-isomorphism $\ph_k:Y_{m_k}\to X$ with $\ep_k'\searrow 0$ and any $i\in \N_n$, Lemma \ref{convlem} for $B_{i,k}:=\{\ph_k(x_{i,m_k})\}$ implies that there exists $\{x_i\}_{i=1}^n\subset X$ such that $\mu_X(\{x_i\})\ge \al_i$ for any $i\in \N_n$. Finally, we prove that $\{x_i\}_{i=1}^n$ are distinct. For any $i\in \N_n$, $x_i$ is the limit of a subsequence of $\{\ph_k(x_{i,m_k})\}_{k\in \N}$ by the construction of $x_i$ in Lemma \ref{convlem}. It follows from the definition of $\{m_k\}_{k\in \N}$ and $1$-Lipschitz continuity of $p_{m_k,m_{k+1}}$ that
\begin{align*}
d_X(x_i,x_j)&\ge \limsup_{k\to \infty} d_X(\ph_k(x_{i,m_k}),\ph_k(x_{j,m_k}))\\
&\ge \limsup_{k\to \infty} d_X(x_{i,m_k},x_{j,m_k})-2\ep_k'\\
&\ge \limsup_{k\to \infty} d_X(x_{i,m_{k-1}},x_{j,m_{k-1}})\\
&\ge d_X(x_{i,m_1},x_{j,m_1}).
\end{align*}
This means that $\{x_i\}_{i=1}^n$ are distinct since $\{x_{i,m_1}\}_{i=1}^n$ are distinct. 
\end{proof}
\section*{Acknowledgment}
The author would like to thank Associate Professor Takumi Yokota for his helpful comments and encouragement. The author also would like to thank Assistant Professor Daisuke Kazukawa for inspiring discussions.

This work was supported by JST SPRING, Grant Number JPMJSP2114.

\end{document}